\documentclass[a4paper,12pt]{article}
\usepackage{amssymb}
\usepackage{latexsym}
\usepackage{amsmath,amsthm,amssymb}
\usepackage{amsfonts}

\newtheorem{tw}{Theorem}[section]

\newtheorem{pro}[tw]{Proposition}
\newtheorem{cor}[tw]{Corollary}

\theoremstyle{definition}

\newtheorem{exa}[tw]{Example}
\newtheorem{rem}[tw]{Remark}

\begin{document}

\begin{center}
{\Large Cram\'er transform and {\bf t}-entropy}
\end{center}
\begin{center}
{\sc Urszula Ostaszewska,\quad Krzysztof Zajkowski}
\footnote{The authors are supported by the Polish National Science Center, Grant no. DEC-2011/01/B/ST1/03838}\\
Institute of Mathematics, University of Bialystok \\ 
Akademicka 2, 15-267 Bialystok, Poland \\ 
uostasze@math.uwb.edu.pl \\ 
kryza@math.uwb.edu.pl 
\end{center}

\begin{abstract}
{\bf t}-entropy is the convex conjugate of the logarithm of the spectral radius of a weighted composition operator (WCO). Let $X$ be a nonnegative random variable.
We show how the Cram\'er transform with respect to the spectral radius  of WCO  is expressed by the {\bf t}-entropy and the Cram\'er transform of the given random variable $X$.
\end{abstract}



\section{Introduction}
Let $M_X$ denote the moment-generating function of a given random variable $X$ that is $M_X(t)=Ee^{tX}$.
A random variable $X$ satisfies the {\it Cram\'er condition} if there exists $c>0$ such that $Ee^{c|X|}<\infty$. If a random variable $X$ satisfies the Cram\'er condition with a constant $c>0$ then $M_X$ is well defined (it takes finite values) on a connected neighborhood, containing the interval
$[-c,c]$, of zero and moreover possesses the following expansion 
$$
M_X(t)= \sum_{n=0}^{\infty}\frac{EX^n}{n!}t^n \ \ \ \textrm{for}\ \ |t|< t_0,
$$
where $t_0\geq c$, compare \cite{Bill}.

The {\it Cram\'er transform} of a random variable $X$ satisfying the Cram\'er condition is the Legendre-Fenchel transform of 
the cumulant generating function of $X$, i.e.
$$
(\ln M_X)^\ast(a)=\sup_{t\in\mathbb{R}}\{at-\ln M_X(t)\}.
$$
It was proved in \cite{DiV} that the following contraction principle holds
\begin{equation}
\label{relE}
(\ln M_X)^\ast(a)=\inf_{m\ll \mu_X,\;\int x dm=a}D(m\Vert \mu_X), 
\end{equation}
where $D(m\Vert \mu_X)=\int\ln\frac{dm}{d\mu_X}dm$ is the {\it relative entropy} of a probability distribution $m$ with respect to the distribution $\mu_X$ of $X$.

Recall now the general notion of the Legendre-Fenchel transform.
Let $f$ be a functional on a real locally convex Hausdorff space $L$ with the values in the extended system of real numbers $\bar{\Bbb R}=[-\infty,+\infty]$.
The set $\mathcal{D}(f)=\{\varphi\in L:\;f(\varphi)<+\infty\}$ is called   the {\it effective domain} of the functional $f$.
The functional $f^\ast:L^\ast\mapsto \bar{\Bbb R}$ that is defined on the dual space by the equality
$$
f^\ast(\mu)=\sup_{\varphi\in L}\{\left\langle\mu,\varphi\right\rangle-f(\varphi)\}=\sup_{\varphi\in \mathcal{D}(f)}\{\left\langle\mu,\varphi\right\rangle-f(\varphi)\}\;\;\;\;\;(\mu\in L^\ast)
$$
is called   the {\it Legendre-Fenchel transform} of the functional $f$ (or the {\it convex conjugate} of $f$).
For a functional $g$ on the dual space $L^\ast$ the Legendre-Fechel transform is defined as the functional
on the initial space given by the similar formula: 
$$
g^{\ast}(\varphi)=\sup_{\mu\in L^\ast}\{\left\langle\mu,\varphi\right\rangle-g(\mu)\}=\sup_{\mu\in \mathcal{D}(g)}\{\left\langle\mu,\varphi\right\rangle-g(\mu)\}\;\;\;\;\;(\varphi\in L).
$$
Let us emphasize that the dual functional $f^\ast$ is convex and lower semicontinuous with respect to the weak-$\ast$ topology on the dual space. 
Moreover, if $f:L\mapsto (-\infty,+\infty]$ is  convex and lower semicontinuous then $(f^\ast)^\ast=f$
(the Legendre-Fenchel transform is involutory).

Now we present a general result obtained
for the spectral radius of weighted composition operators.
Let $\mathcal{X}$ be a  Hausdorff compact space with Borel measure $\mu$, $\alpha:\mathcal{X}\mapsto \mathcal{X}$ a continuous mapping preserving
$\mu$ (i.e. $\mu\circ\alpha^{-1}=\mu$) and $g$ be a continuous function on $\mathcal{X}$.
Antonevich, Bakhtin and Lebedev constructed a functional $\tau_\alpha$ depending upon $\mu$, called ${\bf t}$-entropy (see \cite{ABL2,ABL4}), on the set of probability
and $\alpha$-invariant measures
$\mathcal{M}^1_\alpha$ with values in $[0,+\infty]$ such that for the spectral radius of the weighted composition operator $
(gC_\alpha)u(x)=g(x)u(\alpha(x))$ 
acting in spaces $L^p(\mathcal{X},\mu),\ 1 \leq p < \infty$, the following variational principle holds
\begin{equation}
\label{form2}
\ln r(gC_\alpha)=\max_{\nu\in \mathcal{M}^1_\alpha}
\Big\{\int_\mathcal{X}\ln|g|d\nu-\frac{\tau_\alpha(\nu)}{p}\Big\}.
\end{equation}
It turned out that $\tau_\alpha$ is nonnegative (not necessary taking only finite values), convex and lower semicontinuous on $\mathcal{M}^1_\alpha$.

For  $\varphi \in C(\mathcal{X})$ let  $\lambda(\varphi)=\ln r( e^{\varphi}C_\alpha)$. The functional $\lambda$ is convex and continuous
on $C(\mathcal{X})$ and 
the formula (\ref{form2}) states
that $\lambda$ is the Legendre-Fenchel transform of the function $\frac{\tau_\alpha}{p}$, i.e.
\begin{equation}
\label{vp}
\lambda(\varphi)=\max_{\nu\in \mathcal{M}^1_\alpha}
\Big\{\int_\mathcal{X}\varphi d\nu-\lambda^\ast(\nu)\Big\},
\end{equation}
where
$$
\lambda^{\ast}(\nu)=\left\{ \begin{array}{lcl}
\frac{\tau_\alpha(\nu)}{p} & {\rm for} & \nu\in \mathcal{M}_\alpha^1\;\; {\rm and}\;\;\tau_\alpha(\nu)<+\infty, \\[8pt]   
+\infty  &  {\rm otherwise.}&  
\end{array} \right.\\
$$
It means that the effective domain $\mathcal{D}(\lambda^{\ast})$ is contained in $\mathcal{M}^1_{\alpha}$. 

It turned out that considerations on the spectral exponent (the logarithm of the spectral radius) of some functions of WCO in the natural way lead us   to investigate expressions which are similar to
the cumulant generating functions of random variables (see \cite{OZ3}). Thus it appeared the natural idea to define operators which are moment generating
functions of WCO and next to investigate their spectral exponent using tools related with the Cram\'er transform of given random variables.

This treatment brings together questions which deal with investigations of the spectral radius of some operators and forms of the Cram\'er transform of random variables.

\section{Spectral radius of moment-generating functions of WCO}

A weighted composition operator $e^\varphi C_\alpha$, considered in $L^p$-spaces
(Banach lattices), is an example of positive operators. The spectral radius of any positive operator $A$ belongs to its spectrum (see Prop. 4.1 in Ch. V of \cite{Sch}), i.e. $r(A) \in \sigma(A)$.
Recall that if  $r(A)$ is less than the convergence radius of some analytic function $f$ then one can consider  operators that can be written as   analytic functions of given operators. If the coefficients of $f$ are nonnegative then the composition $f(A)$, for any positive operator $A$, is positive and one has $r(f(A)) \in \sigma(f(A))$. In the following Proposition it is shown that $r(f(A)) = f(r(A))$.

\begin{pro}\label{rfA}
Let $A$ be a positive operator acting in a Banach lattice. Then for any analytic function $f$, with nonnegative coefficients,  such that   
its convergence radius is greater than the spectral radius of $A$ the following holds
$$
r(f(A)) = f(r(A)).
$$
\end{pro} 
\begin{proof}
If the spectrum $\sigma(A)$ of an operator $A$ is contained in the disc of convergence of an analytic function $f$ then one can correctly define
 the operator $f(A)$ and moreover  by the spectral mapping theorem (see for instance \cite{Yos}) we have
\begin{equation}
\label{specan}
\sigma(f(A))=f(\sigma(A)).
\end{equation} 
Since $r(A)\in\sigma(A)$, $f(r(A))\in f(\sigma(A))=\sigma(f(A))$.
Thus we obtain the following inequality
$$
f(r(A)) \leq r(f(A)).
$$  

To obtain the converse one let us consider an arbitrary element $\omega \in \sigma(f(A))$. By (\ref{specan}) there exists $\lambda \in \sigma(A)$ such that $\omega = f(\lambda)$. Obviously $|\lambda | \leq r(A)$ and under the assumption on nonnegativity of coefficients of $f$  we obtain that $f(|\lambda |) \leq f(r(A))$ and consequently
$$
|\omega|\leq f(|\lambda |) \leq f(r(A)).
$$
Recall that $r(f(A)) \in \sigma(f(A))$ and substituting in the above $\omega = r(f(A))$ we have
$$
r(f(A)) \leq f(r(A)).
$$ 
\end{proof}

For a weighted composition operators $e^\varphi C_\alpha$, if $r(e^\varphi C_\alpha)$ is less than the radius of convergence  of $M_X(t)= \sum_{n=0}^{\infty}\frac{EX^n}{n!}t^n$ then one can correctly define an operator 
\begin{equation}
\label{MxA}
M_X(e^\varphi C_\alpha)=\sum_{n=0}^{\infty}\frac{EX^n}{n!}(e^\varphi C_\alpha)^n.
\end{equation}
Assuming $X\ge0$ we have that $EX^n\ge 0$ and by Proposition \ref{rfA} we obtain
$$
r(M_X(e^\varphi C_\alpha))=M_X(r(e^\varphi C_\alpha))=M_X(e^{\lambda(\varphi)}).
$$
Define now a functional
\begin{equation}\label{lambdatilde}
\widetilde{\lambda}_X(\varphi)=
\left\{ \begin{array}{lcl}
(\ln M_X\circ\exp)(\lambda(\varphi)) & {\rm if} & \varphi\in \lambda^{-1}(\mathcal{D}(\ln M_X\circ\exp)), \\[8pt]   
+\infty \qquad \qquad  \qquad  \qquad      & {\rm if\ not}. &  
\end{array} \right.\\
\end{equation}
 Let us emphasize that because $\mathcal{D}(\ln M_X\circ\exp)$ is some left half line or even whole ${\Bbb R}$ and $\lambda$ is a convex functional on $C(\mathcal{X})$ then $\lambda^{-1}(\mathcal{D}(\ln M_X\circ\exp))$ is a convex subset of $C(\mathcal{X})$.
For a nonnegative  random variable $X$ satisfying the Cram\'er condition the cumulant function $\ln M_X$ is convex lower semicontinuous and 
increasing on ${\Bbb R}$. Therefore the composition $\ln M_X\circ\exp$ is convex, lower semicontinuous and also increasing. 
Let us recall that the functional $\lambda$  is convex and continuous on $C(\mathcal{X})$. Then the functional $\widetilde{\lambda}_X$ as a  composition of $\ln M_X \circ \exp$ and $\lambda$ is also convex and lower semicontinuous on $C(\mathcal{X})$. 

Before in Theorem \ref{MTh} we present a form of the convex conjugate of $\widetilde{\lambda}_X$ first we prove Proposition which allow us
characterize the convex conjugate of the composition of some convex functions with the exponent function.

We start with some observations. If $f$ is convex and increasing function then its effective domain $\mathcal{D}(f)$  is some left half line or whole 
$\mathbb{R}$, moreover $\mathcal{D}(f^\ast)\subset [0,+\infty)$.
\begin{pro}
Let $f$ be a convex, increasing and lower semicontinuous function on $\mathbb{R}$ such that $\mathcal{D}(f)$ is some neighborhood of zero.
Then
\begin{equation}
\label{cos}
(f\circ\exp)^\ast(a)=\min_{\alpha\ge 0}\{f^\ast(\alpha)-a\ln \alpha\}+(\exp)^\ast(a)
\end{equation}
for $a\in\mathcal{D}((f\circ\exp)^\ast)$.
\end{pro}
\begin{proof}
Recall that the convex conjugate of the exponent function $\exp^{\ast}(c)$ takes the value $c\ln c - c$ if $c>0$, $\exp^{\ast}(0)=0$ and $\exp^{\ast}(c)= +\infty$ if $c<0$. Because the effective domain of $\exp$ is whole ${\Bbb R}$ then its support function 
$\sigma_{\mathcal{D}(\exp)}(a)= +\infty$
for $a\neq 0$ and $\sigma_{\mathcal{D}(\exp)}(0)=0$.

Observe now that $f\circ\exp$ is convex, increasing and lower semicontinuous. 
For this reason $\mathcal{D}((f\circ\exp)^\ast)\subset [0,+\infty)$.
Using the formula on the convex conjugate of  composite functions 
(see Th. 2.5.1 in \cite{HUiL}), for $a\in\mathcal{D}((f\circ\exp)^\ast)$, we get
\begin{equation}
\label{mini}
(f\circ\exp)^\ast(a)=\min_{\alpha\ge 0}\{f^\ast(\alpha)+(\alpha\exp)^\ast(a)\}.
\end{equation}
Assume first that $a$ is a positive number belonging to $\mathcal{D}((f\circ\exp)^\ast)$. 
If $\alpha =0$ then $(0\exp)^\ast(a)=\sigma_{\mathcal{D}(\exp)}(a)= +\infty$.
It follows that we can search the above minimum for $\alpha >0$. But when $\alpha>0$ then $(\alpha\exp)^\ast(a)= \alpha\exp^\ast (\frac{a}{\alpha}) $. Substituting the formula on $\exp^\ast$ into (\ref{mini}) , for $a > 0$, we obtain 
\begin{equation}
\label{mini1}
(f\circ\exp)^\ast(a) = \min_{\alpha\ge 0}\{f^\ast(\alpha)-a\ln \alpha\}+(\exp)^\ast(a)\\
 \end{equation}

Consider now the possible case when $0\in \mathcal{D}((f\circ\exp)^\ast)$. Notice that then for each $\alpha\ge 0$ $(\alpha\exp)^\ast(0)=0$
and the formula (\ref{mini}) take the form
$$
(f\circ\exp)^\ast(0)=\min_{\alpha\ge 0}f^\ast(\alpha)
$$
that coincides with (\ref{cos}) for $a=0$. On the end let us emphasize that if it is known that $f^\ast$ attains its minimum at a positive number then
we can search the minimum in (\ref{mini1}) for $\alpha > 0$. 
\end{proof}

Observe that if $X$ is a nonnegative and not identically zero (a.e.) random variable  satisfying the Cram\'er condition then its cumulant generating function $\ln M_X$ is convex, increasing and lower semicontinuous. Note that $\ln M_X(0)=0$. Moreover $(\ln M_X)^{\ast}$ attains its minimum at $a=EX$ equals zero. Thus  for the cumulant generating function we can formulate the following 

\begin{cor}
\label{stw}
Let $X$ be a nonnegative and not identically zero (a.e.) random variable  satisfying the Cram\'er condition.
The convex conjugate of the composite function $\ln M_X\circ\exp$ can be expressed by the Cram\'er transform of $X$ as follows
\begin{equation}
\label{form}
(\ln M_X\circ\exp)^\ast(a)=\min_{\alpha> 0}\{(\ln M_X)^\ast(\alpha)-a\ln \alpha\}+(\exp)^\ast(a)
\end{equation}
for $a\in\mathcal{D}((\ln M_X \circ\exp)^\ast)$. Moreover $0\in \mathcal{D}((\ln M_X \circ\exp)^\ast)$ and $(\ln M_X\circ\exp)^\ast(0)=0$. 
\end{cor}
\begin{tw}
\label{MTh}
The convex conjugate of the functional $\widetilde{\lambda}_X$ defined by (\ref{lambdatilde}) is of the form 
\begin{equation}
\label{conj}
\widetilde{\lambda}_X^{\ast}(\widetilde{\nu})=\frac{1}{p}\widetilde{\nu}(\mathcal{X})\tau_\alpha\Big(\frac{\widetilde{\nu}}{\widetilde{\nu}(\mathcal{X})}\Big)+ (\ln M_X\circ\exp)^\ast(\widetilde{\nu}(\mathcal{X})).
\end{equation}
\\ If $\widetilde{\nu}(\mathcal{X}) =0$ then $\widetilde{\lambda}_X^{\ast}({\bf 0})=0$.  And the  effective domain of $\widetilde{\lambda}_X^{\ast}$ is contained in the set 
$ 
\widetilde{\mathcal{M}}=\{\widetilde{\nu}=a\nu : \nu \in \mathcal{M}^1_{\alpha}\ {\textrm and} \ a \in \mathcal{D}((\ln M_X\circ\exp)^{\ast}) \}
$.
\end{tw}
\begin{proof}
The composition $\ln M_X\circ\exp$ is convex and lower semicontinuous. 
By the involutory of the Legendre-Fenchel transform we get
\begin{equation}
\label{vpcomp}
(\ln M_X\circ\exp)(t)
 =  \sup_{a\in\mathcal{D}((\ln M_X\circ\exp)^{\ast})}\Big\{t a - (\ln M_X\circ\exp)^\ast(a)\Big\}.
\end{equation}
Since for $a=0$ the expression on the right hand side is equal zero the supremum can be search on the set $\mathcal{D}((\ln M_X\circ\exp)^{\ast})\setminus \{0\}$.
 
Substituting $t=\lambda(\varphi)$ into (\ref{vpcomp})  and using the variational principle (\ref{vp}) we get
\begin{eqnarray*}
\widetilde{\lambda}_X(\varphi)
& = & \sup_{a\in\mathcal{D}((\ln M_X\circ\exp)^{\ast})\setminus \{0\}}\Big\{\lambda(\varphi) a - (\ln M_X\circ\exp)^\ast(a)\Big\}\\
\: &=& \sup_{a\in\mathcal{D}((\ln M_X\circ\exp)^{\ast})\setminus \{0\}}\sup_{\nu\in \mathcal{M}^1_\alpha}\Big\{\int_\mathcal{X}\varphi d(a\nu)-a\frac{\tau_\alpha(\nu)}{p}- (\ln M_X\circ\exp)^\ast(a)\Big\}.\\
\end{eqnarray*} 
Denoting $a\nu$ by $\widetilde{\nu}$ we have that $\widetilde{\nu}(\mathcal{X})=a$ and $\nu=\frac{\widetilde{\nu}}{\widetilde{\nu}(\mathcal{X})}$ for $\widetilde{\nu}(\mathcal{X})\neq 0$.
Let us define $\widetilde{\mathcal{M}}_{+}=\{a\nu:\; \nu\in \mathcal{M}^1_\alpha\;{\textrm and}\;a\in\mathcal{D}((\ln M_X\circ\exp)^{\ast})\setminus \{0\}\}$. Note that $\widetilde{\mathcal{M}}_{+}=\widetilde{\mathcal{M}}\setminus \{\bf{0}\}$.
Applying the introduced notations we can rewrite the above as follows
$$
\widetilde{\lambda}_X(\varphi)=\sup_{\widetilde{\nu}\in \widetilde{\mathcal{M}}_{+}}\Big\{\int_\mathcal{X}\varphi d\widetilde{\nu}-\frac{1}{p}\widetilde{\nu}(\mathcal{X})\tau_\alpha\Big(\frac{\widetilde{\nu}}{\widetilde{\nu}(\mathcal{X})}\Big)- (\ln M_X\circ\exp)^\ast(\widetilde{\nu}(\mathcal{X}))\Big\}.
$$
Let us note that the above equation has the form of the Legendre-Fenchel transform. Thus we immediately obtain convexity and lower semicontinuity of the functional $\widetilde{\lambda}_X$ on $C(\mathcal{X})$.

It remains to prove that the expression
\begin{equation}\label{suma}
\frac{1}{p}\widetilde{\nu}(\mathcal{X})\tau_\alpha\Big(\frac{\widetilde{\nu}}{\widetilde{\nu}(\mathcal{X})}\Big)+ (\ln M_X\circ\exp)^\ast(\widetilde{\nu}(\mathcal{X}))
\end{equation}
is convex and lower semicontinuos on $\widetilde{\mathcal{M}}_{+}$.
Notice now that $\widetilde{\mathcal{M}}$ is some subset (convex subset) of $C(\mathcal{X})^\ast$ and $\widetilde{\nu}(\mathcal{X})$ is the total variation of $\widetilde{\nu}$
on $\widetilde{\mathcal{M}}$ that is a norm on $C(\mathcal{X})^\ast$. For this reason the functions $\widetilde{\nu}\mapsto\widetilde{\nu}(\mathcal{X})$ and
$\widetilde{\nu}\mapsto\frac{\widetilde{\nu}}{\widetilde{\nu}(\mathcal{X})}$ are continuous on $\widetilde{\mathcal{M}}_{+}$. 
The ${\bf t}$-entropy and $(\ln M_X\circ \exp)^\ast$ are lower semicontinuous on $\mathcal{M}^1_\alpha$ and $\mathbb{R}$, respectively. Thus 
the expression (\ref{suma}) is lower semicontinuous on $\widetilde{\mathcal{M}}_{+}$. 

Convexity of $(\ln M_X\circ\exp)^\ast$ on $\mathbb{R}$, additivity and positive homogeneity
of the total variation on $\widetilde{\mathcal{M}}$ gives convexity of $(\ln M_X\circ \exp)^\ast(\widetilde{\nu}(\mathcal{X}))$ on $\widetilde{\mathcal{M}}$.
Moreover by convexity of $\tau_{\alpha}$, for $s \in [0,1]$, we get
$$ 
[s\widetilde{\nu}_1(\mathcal{X})+ (1-s)\widetilde{\nu}_2(\mathcal{X})] \tau_{\alpha}\Big( \frac{s\widetilde{\nu}_1+ (1-s)\widetilde{\nu}_2}{s\widetilde{\nu}_1(\mathcal{X})+ (1-s)\widetilde{\nu}_2(\mathcal{X})}\Big)
$$
$$
= [s\widetilde{\nu}_1(\mathcal{X})+ (1-s)\widetilde{\nu}_2(\mathcal{X})]\tau_{\alpha}
\Big(\frac{s\widetilde{\nu}_1(\mathcal{X})}{s\widetilde{\nu}_1(\mathcal{X})+ (1-s)\widetilde{\nu}_2(\mathcal{X})}\cdot\frac{\widetilde{\nu}_1}{\widetilde{\nu}_1(\mathcal{X}) }+ \frac{(1-s)\widetilde{\nu}_2(\mathcal{X})}{s\widetilde{\nu}_1(\mathcal{X})+ (1-s)\widetilde{\nu}_2(\mathcal{X})}\cdot\frac{\widetilde{\nu}_2}{\widetilde{\nu}_2(\mathcal{X})}\Big) 
$$
$$
\leq s\widetilde{\nu}_1(\mathcal{X})\tau_{\alpha}\Big(\frac{\widetilde{\nu}_1}{\widetilde{\nu}_1(\mathcal{X}) }\Big) + (1-s)\widetilde{\nu}_2(\mathcal{X})\tau_{\alpha}\Big(\frac{\widetilde{\nu}_2}{\widetilde{\nu}_2(\mathcal{X}) }\Big).
$$
For this reason the expression (\ref{suma}) is convex and lower semicontinuous on $\widetilde{\mathcal{M}}_{+}$. It means that the formula (\ref{suma}) is equal to $\widetilde{\lambda}^{\ast}$ on this set. 

To calculate the value of $\widetilde{\lambda}^{\ast}$ at $\widetilde{\nu}\equiv {\bf 0}$ we use the Legendre-Fenchel transform, i.e. 
$$ 
\widetilde{\lambda}_X^{\ast}({\bf 0}) =\sup_{\varphi \in \mathcal{D}(\widetilde{\lambda}_X)}\{-(\ln M_X\circ\exp)(\lambda(\varphi))\}=-\inf_{\varphi \in \mathcal{D}(\widetilde{\lambda}_X)}(\ln M_X)(r(e^\varphi T_\alpha)).
$$
The cumulant generating function $\ln M_X$ is continuous at $0$ and its value equals $0$. Because the spectral radius $r(e^\varphi T_\alpha)$ can be an arbitrary
small positive number then we obtain that $\widetilde{\lambda}_X^\ast({\bf 0})=0$.
\end{proof}

\begin{exa}
Let a random variable $X$ be exponentially distributed with a positive parameter $\mu$, i.e. with the density function
$
f(x)=\mu e ^{-\mu x}{\bf 1}_{(0,\infty)}(x).  
$  
Its cumulant generating function is 
$
\ln M_X(t)= \ln\frac{\mu}{\mu - t}
$
for $t<\mu$ and $+\infty$ otherwise. It is a convex and increasing function. The classical Legendre-Fenchel transform gives that 
$$
(\ln M_X)^\ast(a)=a[(\ln M_X)']^{-1}(a)-(\ln M_X)([(\ln M_X)']^{-1}(a)),
$$
where $[(\ln M_X)']^{-1}$ is the inverse function to the derivative of $\ln M_X$. By direct calculations we get that
$$
(\ln M_X)^\ast(a)=\mu a-\ln(\mu a)-1
$$
for $a>0$.
In the same manner we can obtain the formula on
\begin{equation}
\label{comp}
(\ln M_X\circ\exp)^\ast(a)=a\ln(\mu a)-(a+1)\ln(a+1)\quad (a>0).
\end{equation}
We consider the operator of the form $M_X(A)= \mu (\mu I- A)^{-1}$ which is well defined if $r(A)<\mu$. For the weighted composition $e^\varphi C_\alpha$, the set
$\{\varphi\in C(\mathcal{X}):\ r(e^\varphi C_\alpha)<\mu\}$ is the effective domain of the functional $\widetilde{\lambda}_X$.
Substituting the formula on $(\ln M_X\circ\exp)^\ast$ into (\ref{conj}) we get the evident form of the convex conjugate of $\widetilde{\lambda}_X$ for the exponentially distributed random variable.
\end{exa}
\begin{rem}
Using the Legendre-Fenchel transform we can also obtain the formula
$$
((\ln M_X)^\ast\circ\exp)^\ast(a)=(a+1)\ln(a+1)-a\ln\mu -a.
$$
In this case $\mathcal{D}((\ln M_X\circ\exp)^\ast)=[0,\infty)$.
\end{rem}

Recall that if $X$ satisfies the Cram\'er condition with some $c>0$ then for $t\in (-c,c)$ one has
$$
M_X(t)= \sum_{n=0}^{\infty}\frac{EX^n}{n!}t^n
$$ 
and it is known that this power series possesses the convergence radius $R$ not less than $c$. 
Moreover if the moment-generating function $M_X$ of a nonnegative random variable satisfies  additionally  condition $\lim_{t\to R^{-}} M_X(t)=+\infty$ then, by  Theorem 2.5 in  \cite{OZ3},
we obtain the following formula on the convex conjugate of composition $\ln M_X\circ\exp$ depending on the moments of random variable $X$ 
\begin{equation}
\label{form3}
(\ln M_X \circ \exp)^{\ast}(a)=\left\{ 
\begin{array}{ll}
\min\limits_{(t_k)\in S_a} \liminf\limits_{N\to \infty} \sum\limits_{k=0}^{N} t_k\ln \frac{t_k k!}{EX^k} & a>0, \\
0 & a=0,\\
+\infty & a<0,
\end{array} \right.
\end{equation} 
where $S_a=\{ (t_n): t_n\geq 0,\ \sum_{n=0}^{\infty} t_n=1 \ {\rm and} \ \sum_{n=0}^{\infty} n t_n =a\}$. 
Let us emphasize that it is an another formula on the  convex conjugate of the composition $\ln M_X\circ\exp$.

\begin{exa}
The moments of the exponentially distributed random variables $X$  equal
$E X^n = \frac{n!}{\mu^n}$ for any $n$. Notice that the moment-generating function of $X$ satisfies assumptions of Theorem 2.5 in  \cite{OZ3} and for $a>0$, by the formula (\ref{form3}), we obtain
\begin{eqnarray}
\label{momenty}
(\ln M_X \circ \exp)^{\ast}(a)&=& \min\limits_{(t_n)\in S_a} \liminf\limits_{N\to \infty} \sum_{k=0}^{N} (k t_k \ln \mu+t_k\ln t_k) \nonumber \\
&=&\min\limits_{(t_n)\in S_a}\Big\{\Big(\sum_{n=0}^\infty nt_n\Big)\ln  \mu +  \sum_{n=0}^\infty t_n\ln t_n\Big\} \nonumber \\
&=& a\ln  \mu + \min_{(t_n)\in S_a}  \sum_{n=0}^\infty t_n\ln t_n .
\end{eqnarray} 
Let us emphasize that, how it was proved in \cite[Prop. 2.3]{Zaj}, the entropy function of infinite numbers of variables $\sum_{n=0}^\infty t_n\ln t_n$ takes
on the set $\{ (t_n): t_n\geq 0,\ \sum_{n=0}^{\infty} t_n=1 \ {\rm and} \ \sum_{n=0}^{\infty} n t_n <\infty\}$
finite values and therefore the above series are convergent.
Moreover  comparing (\ref{comp})  and  (\ref{momenty}) we get 
$$
a \ln a - (a+1)\ln(a+1)= \min_{(t_n)\in S_a} \sum_{n=0}^\infty t_n\ln t_n .
$$
On the left handside there is the Legendre-Fenchel transform of $\ln M_X \circ \exp$ for the parameter $\mu=1$.

\end{exa}

Consider a discrete random variable $X$ taking values in ${\Bbb N}\cup \{0\}$; $P(X=n)=p_n$.
In this case it appears another opportunity of an application of Theorem 2.5 \cite{OZ3}.
The probability-generating function of $X$ has the form 
$$
g_X(s)= \sum_{n=0}^{\infty} p_n s^n.
$$
Since $g_X(1)=1$, the convergence radius $R$ of $g_X$ is not less than $1$.
Let $A$ be a positive operator with the spectral radius less than $R$ and greater than zero. We can consider now an operator $g_X(A)$. Let us emphasize that it is new different kind of operator than above considered. Its spectral radius can be rewritten as follows
\begin{eqnarray*}
r(g_X(A)) &=&  \sum_{n=0}^{\infty} p_n r(A)^n = \sum_{n=0}^{\infty} p_n e^{n \ln r(A)}\\
&=& M_X(\ln r(A)).
\end{eqnarray*}

If the operator $A$ is a weighted composition operator then we obtain that the logarithm of the spectral radius of 
\begin{equation}
\label{gxA}
g_X(e^{\varphi}C_{\alpha})=\sum_{n=0}^{\infty} p_n (e^{\varphi}C_{\alpha})^n
\end{equation} 
is  a composition of cumulant and functional $\lambda$, i.e. 
$$
\ln r(g_X(e^{\varphi}C_{\alpha}))= (\ln M_X \circ \lambda)(\varphi).
$$ 
Define now a functional $\widehat{\lambda}_X$  by the following formula
\begin{equation}
\label{tworz}
\widehat{\lambda}_X(\varphi)= (\ln M_X \circ \lambda)(\varphi)
\end{equation}
for $\varphi\in\lambda^{-1}(\mathcal{D}(\ln M_X))$ and $+\infty$ otherwise.

Because the cumulant generating function is convex and lower semicontinuous on $\mathbb{R}$
then the following equality is satisfied
\begin{equation}
\label{vpCum}
(\ln M_X)(t) =  \sup_{a\in\mathbb{R}}\Big\{t a - (\ln M_X)^\ast(a)\Big\}.
\end{equation}
Observe that for the discrete random variable with values in the set of nonnegative integers  
$$
(\ln M_X)(t)=(\ln g_X\circ\exp)(t)=\ln\sum_{n=0}^\infty p_ne^{nt}
$$
and we can use once again Theorem 2.5 in  \cite{OZ3}. Taking  in Theorem 2.5 $a_n$ equals the probability $p_n$, assuming that $p_n>0$, we obtain the following formula
$$
(\ln M_X)^\ast(a)=\min_{(t_k)\in S_a} \liminf_{N\to \infty} \sum_{k=0}^{N} t_k\ln \frac{t_k}{p_k}\quad (a\in int\mathcal{D}((\ln M_X)^\ast)) 
$$
which is an example (for a discrete random variable) of the contraction principle (\ref{relE}). 

Substituting $t=\lambda(\varphi)$ into (\ref{vpCum})
and using the formula (\ref{vp}) we get
$$
\widehat{\lambda}_X(\varphi)=\sup_{a\in \mathcal{D}((\ln M_X)^\ast)}\sup_{\nu\in \mathcal{M}^1_\alpha}\Big\{\int_\mathcal{X}\varphi d(a\nu)-a\frac{\tau_\alpha(\nu)}{p}- (\ln M_X)^\ast(a)\Big\}.
$$
Defining now the set $\widehat{\mathcal{M}}_+=\{a\nu:\; \nu\in \mathcal{M}^1_\alpha\;{\textrm and}\;a\in \mathcal{D}((\ln M_X)^\ast)\}\setminus \{{\bf 0}\}$ and introducing it to the above formula
we get
$$
\widehat{\lambda}_X(\varphi)= \sup_{\widehat{\nu}\in \widehat{\mathcal{M}}_+}\Big\{  \int_\mathcal{X}\varphi d\widehat{\nu} - \frac{1}{p}\widehat{\nu}(\mathcal{X})\tau_\alpha\Big(\frac{\widehat{\nu}}{\widehat{\nu}(\mathcal{X})}\Big)- (\ln M_X)^\ast(\widehat{\nu}(\mathcal{X}))\Big\}.
$$
The expression
$$
\frac{1}{p}\widehat{\nu}(\mathcal{X})\tau_\alpha\Big(\frac{\widehat{\nu}}{\widehat{\nu}(\mathcal{X})}\Big)+ (\ln M_X)^\ast(\widehat{\nu}(\mathcal{X}))
$$
is convex and lower semicontinuous on $\widehat{\mathcal{M}}_+$ and
\begin{eqnarray*} 
\widehat{\lambda}_X^{\ast}({\bf 0}) & = & -\inf_{\varphi \in C(X)}(\ln M_X)(\lambda(\varphi))\\
\; & = & -\inf_{\varphi \in C(X)}\ln\sum_{n=0}^\infty p_n r(e^\varphi C_\alpha)^n = -\ln p_0. 
\end{eqnarray*}
In this way we obtained the following
\begin{pro}
\label{MPr}
For the functional $\widehat{\lambda}_X$ given by (\ref{tworz})  the following variational principle holds
$$
\widehat{\lambda}_X(\varphi)= \sup_{\widehat{\nu}\in \widehat{\mathcal{M}}}\Big\{  \int_\mathcal{X}\varphi d\widehat{\nu} - \widehat{\lambda}_X^{\ast}(\widehat{\nu})\Big\},
$$
where $  \widehat{\mathcal{M}}=\{\widehat{\nu}=a \nu:\ \nu \in \mathcal{M}^1_{\alpha}\ {\textrm and} \ a\in \mathcal{D}((\ln M_X)^{\ast})  \}$ and
$$
\widehat{\lambda}_X^{\ast}(\widehat{\nu})=\frac{1}{p}\widehat{\nu}(\mathcal{X})\tau_\alpha\Big(\frac{\widehat{\nu}}{\widehat{\nu}(\mathcal{X})}\Big)+ (\ln M_X)^\ast(\widehat{\nu}(\mathcal{X})) \ \ {\textrm for} \ \ \widehat{\nu}(\mathcal{X})>0.
$$ 
If $\widehat{\nu}(\mathcal{X})=0$ then $\widehat{\lambda}^{\ast}_X({\bf 0})=-\ln p_0$. 
\end{pro}

\begin{rem}
Let us stress once again that Theorem \ref{MTh} and Proposition \ref{MPr} are dealt with two different classes of operators. 
In the first one we consider operators
that can be symbolically written as $\int_0^\infty e^{sA}\mu_X(ds)$, where $A=e^\varphi C_\alpha$ and the integral is understood  in the sens of 
the power series (\ref{MxA}). In the second one we investigate the spectral exponent of operators of the form $\int_0^\infty A^s\mu_X(ds)$, where this integral is defined by the series (\ref{gxA}). 
Therefore Proposition \ref{MPr} is not a simple subcase of Theorem \ref{MTh}. For a discrete random variables with values in $\mathbb{N}\cup \{0\}$ we can defined a new type of considered operators. 
\end{rem}

\begin{exa}
For Poisson distributed  $X$ with parameter $\mu$ the probability-generating function is of the form 
$$
g_X(s)=e^{\mu(s-1)}.
$$
Then the operator $g_X(A)=e^{-\mu}e^{\mu A}$ and 
\begin{equation}\label{loggXA}
\ln r(g_X(A))= \ln (e^{-\mu}e^{\mu r(A)})= \mu r(A) -\mu= \ln M_X(\ln r(A)). 
\end{equation}
The cumulant generating function of $X$ is equal to $\ln M_X(t)= \mu  e^t - \mu$ and its Cram\'er transform  
has the form 
\begin{equation*}
(\ln M_X )^{\ast}(a)=\left\{ 
\begin{array}{ll}
\mu - a + a \ln\frac{a}{\mu} & a>0, \\
\mu & a=0,\\
+\infty & a<0.
\end{array} \right.
\end{equation*} 

Taking $A=e^{\varphi}C_{\alpha}$ in (\ref{loggXA}) by Proposition \ref{MPr} for $\widetilde{\nu}(\mathcal{X})>0$ we obtain that
$$
\widehat{\lambda}_X^{\ast}(\widehat{\nu})=\frac{1}{p}\widehat{\nu}(\mathcal{X})\tau_\alpha\Big(\frac{\widehat{\nu}}{\widehat{\nu}(\mathcal{X})}\Big)+ \mu - \widehat{\nu}(\mathcal{X}) + \widehat{\nu}(\mathcal{X}) \ln\frac{\widehat{\nu}(\mathcal{X})}{\mu}.
$$ 
If we consider now the operator $M_X(A)= e^{-\mu} e^{\mu e^A}$ then for $A=e^{\varphi}C_{\alpha}$ by  Theorem \ref{MTh} for $\widetilde{\nu}(\mathcal{X})>0$
we get 
$$
\widetilde{\lambda}_X^{\ast}(\widetilde{\nu})=\frac{1}{p}\widetilde{\nu}(\mathcal{X})\tau_\alpha\Big(\frac{\widetilde{\nu}}{\widetilde{\nu}(\mathcal{X})}\Big)+ (\ln M_X\circ\exp)^\ast(\widetilde{\nu}(\mathcal{X})),
$$
where
$$
(\ln M_X\circ\exp)^\ast(\widetilde{\nu}(\mathcal{X}))= \min_{\alpha> 0} \Big\{ \mu - \alpha + \alpha\ln \frac{\alpha}{\mu} - \widetilde{\nu}(\mathcal{X}) \ln \alpha \Big\} + (\exp)^{\ast} (\widetilde{\nu}(\mathcal{X})).
$$
\end{exa}

{\bf Acknowledgment.} We would like to thank the reviewer for his very precise improvements and inquiring remarks and comments.


\end{document}